\documentclass[ppn]{svjour}
\usepackage{color,graphicx,esint,amsfonts,setspace}
\usepackage[colorlinks=true]{hyperref}
\hypersetup{urlcolor=blue, citecolor=red, linkcolor=blue}

\newcommand{\be}[1]{\begin{equation}\label{#1}}
\newcommand{\ee}{\end{equation}}
\renewcommand{\(}{\left(}
\renewcommand{\)}{\right)}

\newcommand{\Email}[1]{E-mail: \href{mailto:#1}{\textsf{#1}}}
\newcommand{\eqref}[1]{(\ref{#1})}
\newcommand{\vertiii}[1]{{\left\vert\kern-0.25ex\left\vert\kern-0.25ex\left\vert #1\right\vert\kern-0.25ex\right\vert\kern-0.25ex\right\vert}}

\newcommand{\R}{{\mathbb R}}
\renewcommand{\S}{{\mathbb S}}

\newcommand{\Z}{{\mathbb Z}}

\newcommand{\nA}{\nabla_{\mathbf A}}
\newcommand{\LapA}{\Delta_{\mathbf A}}
\newcommand{\HA}{\mathcal H(\R^2)}
\newcommand{\EA}{E_{\mathbf A}}
\newcommand{\nrm}[2]{\|{#1}\|_{#2}}
\newcommand{\nrmS}[2]{\|{#1}\|_{\mathrm L^{#2}(\S^1,d\sigma)}}

\newcommand{\nrmdeux}[2]{\left\|{#1}\right\|_{\mathrm L^{#2}(\R^2)}}

\newcommand{\iS}[1]{\int_{\S^1}{#1}\,d\sigma}

\newcommand{\iC}[1]{\iint_{\R\times\S^1}{#1}\,ds\,d\sigma}
\newcommand{\ir}[2]{\int_{\R^{#1}}{#2}\,dx}
\newcommand{\is}[1]{\int_{\R}{#1}\,ds}

\newcommand{\iz}[1]{\int_{-1}^{+1}{#1}\,dz}
\newcommand{\CKNa}{\mathrm{A}}
\newcommand{\CKNb}{\mathrm{B}}

\begin{document}

\title{Symmetry results in two-dimensional inequalities for Aharonov-Bohm magnetic fields}
\titlerunning{Symmetry and Aharonov-Bohm magnetic fields}

\author{Denis Bonheure,\kern-3pt\inst1 Jean Dolbeault,\kern-3pt\inst2 Maria J.~Esteban,\kern-3pt\inst2 Ari Laptev,\kern-3pt\inst3 Michael Loss\inst4}
\authorrunning{D.~Bonheure, J.~Dolbeault, M.J.~Esteban, A.~Laptev, \& M.~Loss}

\institute{D\'epartement de math\'ematique, Facult\'e des sciences, Universit\'e Libre de Bruxelles, Campus Plaine CP 213, Bld du Triomphe, B-1050 Brussels, Belgium. \Email{dbonheur@ulb.ac.be}\and
CEREMADE (CNRS UMR n$^\circ$ 7534), PSL university, Universit\'e Paris-Dauphine, Place de Lattre de Tassigny, 75775 Paris 16, France. \Email{dolbeaul@ceremade.dauphine.fr} (J.D.), \Email{esteban@ceremade.dauphine.fr} (M.J.E)\and
Department of Mathematics, Imperial College London, Huxley Building, 180 Queen's Gate, London SW7 2AZ, UK. \Email{a.laptev@imperial.ac.uk}\and
School of Mathematics, Skiles Building, Georgia Institute of Technology, Atlanta GA 30332-0160, USA. \Email{loss@math.gatech.edu}}

\maketitle

\begin{abstract} This paper is devoted to the symmetry and symmetry breaking properties of a two-dimensional magnetic Schr\"odinger operator involving an Aharonov-Bohm magnetic vector potential. We investigate the symmetry properties of the optimal potential for the corresponding magnetic Keller-Lieb-Thir\-ring inequality. We prove that this potential is radially symmetric if the intensity of the magnetic field is below an explicit threshold, while symmetry is broken above a second threshold corresponding to a higher magnetic field. The method relies on the study of the magnetic kinetic energy of the wave function and amounts to study the symmetry properties of the optimal functions in a magnetic Hardy-Sobolev interpolation inequality. We give a quantified range of symmetry by a non-perturbative method. To establish the symmetry breaking range, we exploit the coupling of the phase and of the modulus and also obtain a quantitative result. \end{abstract}

breaking; magnetic Hardy-Sobolev inequality; magnetic interpolation inequality; optimal constants; magnetic Schr\"odinger operator; magnetic rings; Caffarelli-Kohn-Nirenberg inequalities; Hardy-Sobolev inequalities.

\bigskip

\section{Introduction and main results}\label{Sec:Introduction}

It is a basic question in the calculus of variations whether an optimizer possesses the same symmetries as the minimized functional. The answer depends very much on the details of the functional. The only examples where a general theory is available are linear variational problems where the optimizers are eigenfunctions and hence belong to an irreducible representation of the symmetry group.

Apart when uniqueness is known, for instance in strictly convex problems, it is very difficult to determine whether the minimizer of a non-linear functional is symmetric or whether the symmetry is broken. Symmetry breaking can be shown by minimizing the functional in the class of symmetric functions and then by considering the second variation around this state. If the resulting quadratic form has a negative eigenvalue, then the symmetry is broken. Such a computation can be a formidable problem since it requires quite a bit of information about the minimizer in the symmetric class. Thus, at least in principle, there is a systematic method for proving symmetry breaking. Needless to say, the positivity of the lowest eigenvalue only indicates the stability of the minimizer in the symmetric class but has no bearing on the symmetry of the true minimizer. In general, if the symmetry of the true minimizer is broken, then it does not belong to an irreducible representation of the symmetry group, although it may be invariant under a subgroup. It is therefore difficult to describe minimizers whose symmetry is broken.

It should be clear from these remarks that there is no general theory for proving symmetry and one has to consider some basic examples. There has been a number of non-quadratic variational problems such as Sobolev-type and Hardy-Littlewood-Sobolev inequalities for which this question was successfully answered~\cite{Aubin-76,Talenti-76,Lieb-83,MR1940370}. The techniques relied mostly on rearrangement inequalities that are closely related to the isoperimetric problem. Another interesting class of examples is furnished by the Caffarelli-Kohn-Nirenberg inequalities which display optimizers with symmetry or with broken symmetry, depending on the values of certain parameters~\cite{Catrina-Wang-01,MR2001882,Felli-Schneider-03}. The method for determining the symmetry range, however, is entirely different from the above mentioned techniques and proceeds through a flow method~\cite{MR3570296,MR3612700}.

Common to all these problems is that these are functionals acting on scalar, possibly positive functions. A class of problems that does not fit this mould involve external magnetic fields. In this case the wave function, out of necessity, is truly complex valued and hence the Euler-Lagrange equations form a system of partial differential equations. There has been a number of results for constant magnetic fields. L.~Erd\"os in~\cite{Erdos96} proved symmetry in a Faber-Krahn type inequality and in~\cite{BONHEURE2018} D.~Bonheure, M.~Nys and J.~Van Schaftingen proved symmetry for some non-linear problems involving small, constant magnetic fields, albeit in a perturbative sense. Some estimates in~\cite{MR3784917} suggest that symmetry can also be expected in non-perturbative regimes as well. Likewise the ground state of a quantum particle confined to a circle with a non-zero magnetic flux is treated in~\cite{doi:10.1063/1.5022121} and optimal results for symmetry and symmetry breaking are given there.

\medskip In this paper we treat a non-linear problem related to a Hardy inequality due to A.~Laptev and T.~Weidl. Let us consider an Aharonov-Bohm vector potential
\be{eq:AB-vect}
\mathbf A(x)=\frac a{|x|^2}\(x_2,-\,x_1\)\,,\quad x=(x_1,x_2)\in\R^2\setminus\{0\}
\ee
corresponding to a singular magnetic field of intensity proportional to $a\in\R$. 
It was shown in~\cite{MR1708811} that
\be{LaWe}
\ir2{|\nA\psi|^2}\ge\min_{k\in\mathbb Z}\,(a-k)^2\ir2{\frac{|\psi|^2}{|x|^2}}\,,\quad\mbox{where}\quad\nA\psi := \nabla\psi+\,i\,\mathbf A\,\psi\,,
\ee
and that there is no function for which there is equality. The interesting point is that there is no Hardy inequality in two dimensions without magnetic field and a non-trivial magnetic field is therefore crucial for the Hardy inequality to hold. 

The main result of this paper is concerned with a generalization of~\eqref{LaWe} to the magnetic interpolation inequality
\be{Ineq:MgnHS}
\ir2{|\nA\psi|^2}+\lambda\ir2{\frac{|\psi|^2}{|x|^2}}\ge\mu(\lambda)\(\ir2{\frac{|\psi|^p}{|x|^2}}\)^{2/p} \,.
\ee
We shall give a range for the parameters $a$, $\lambda$ and $p$ for which the minimizer is symmetric and compute the minimizer explicitly. We emphasize that this range is quantitative and is not based on perturbation theory. Moreover, we give also a quantitative range for the parameters for which the symmetry is broken, see Theorem~\ref{Thm:MgnHS}. In the definition of the vector potential $\mathbf A$, the constant $a$ could be replaced by a function $a(\theta)$, where $\theta$ is the angle in polar coordinates. See Section~\ref{Sec:MagneticField}.

As mentioned before, the main difficulty is that the function $\psi$ is complex valued, \emph{i.e.}, at least when $\psi\neq0$, it can be written in the form $\psi = |\psi|\,e^{iS}$ where the phase $S(x)$ is non-trivial. In such a context, standard techniques such as symmetrization have shown to be successful only in very particular situations, see e.g.~\cite{Erdos96} and~\cite{MR3665549,2018arXiv180506294L}. To explain some of the ideas involved, we use polar coordinates and, when $\psi$ does not vanish, write
\[
\ir2{|\nA\psi|^2} = \ir2{\left[( \partial_r\,|\psi|)^2 + (\partial_r S)^2\,|\psi|^2 + \frac{1}{r^2}\,(\partial_\theta S+A)^2\,|\psi|^2\right]}\,.
\]
By dropping the term involving $\partial_r S$ the inequality effectively reduces to a problem in which the phase $S$ depends only on the polar angle. By optimizing over the phase using ideas from~\cite{doi:10.1063/1.5022121} the problem is then brought into a form that is a particular class of Caffarelli-Kohn-Nirenberg inequalities for which detailed results are known~\cite{MR3570296}. The symmetry breaking part is more complicated. The chief reason for this is that the term involving $\partial_r S$ has to be taken into account. This leads to a rather involved computation yielding a region for symmetry breaking that, however, is surprisingly close to the complement of the region where symmetry holds. Good control of the interplay between the phase of $\psi$ and its modulus is a key point to obtaining good estimates for the parameter where symmetry breaking occurs.

\medskip There are some interesting consequences of this result. The first is a magnetic version of the Keller-Lieb-Thirring inequalities. Let $q\in(1,+\infty)$ and define the weighted norm
\[
\vertiii\phi_q:=\(\ir2{|\phi|^q\,|x|^{2\,(q-1)}}\)^{1/q}\,.
\]
We denote by $\mathrm L^q_\star(\R^2)$ the space of measurable functions such that $\vertiii\phi_q$ is finite. Our first result is an estimate on the ground state energy $\lambda_1$ of the magnetic Schr\"odinger operator $-\LapA-\phi$ on $\R^2$ and a symmetry result of the corresponding optimal potential $\phi$. The magnetic Schr\"odinger energy is
\[
\EA[\psi]:=\ir2{\Big(|\nA\psi|^2-\phi\,|\psi|^2\Big)}\quad\mbox{where}\quad\nA\psi:=\nabla\psi+\,i\,\mathbf A\,\psi
\]
and $\lambda_1(-\LapA-\phi)$ is defined as the infimum of $\EA[\psi]$ on all $\psi\in\HA$ such that $\ir2{|\psi|^2\,|x|^{-2}}=1$, where $\HA$ denotes the homogeneous $\mathrm H^1$ space, \emph{i.e.},
\[
\HA=\left\{u\in\mathrm L^1_{\rm loc}(\R^2\setminus\{0\},\mathbb C)\,:\,|x|^{-1}\,u\in\mathrm L^2(\R^2,\mathbb C),\,\nabla u\in\mathrm L^2(\R^2,\mathbb C^2)\right\}\,.
\]
\begin{theorem}[A magnetic Keller-Lieb-Thirring estimate]\label{Thm:GroundState} Let $a\in\R$, $q\in(1,\infty)$ and $\phi\in\mathrm L^q_\star(\R^2)$. Then there is a convex monotone increasing function $\mu\mapsto\lambda(\mu)$ on~$\R^+$ such that $\,\lim_{\mu\to 0^+} \lambda(\mu)= - \min_{k\in \Z} (a-k)^2 $ and
\be{KLT}
\lambda_1(-\LapA-\phi)\ge-\,\lambda\(\vertiii\phi_q\)\,.
\ee
For $a\in (0, 1/2)$, there is an explicit value $\mu_\star=\mu_\star(a)>0$ such that the potential
\[
\phi(x)=\(|x|^\alpha+|x|^{-\alpha}\)^{-2}\quad\forall\,x\in\R^2\,,\quad\mbox{with}\quad\alpha=\frac{p-2}2\,\sqrt{\lambda(\mu)+a^2}\,,
\]
is optimal for any $\mu\le\mu_\star$. On the contrary, for all $a\in (0, 1/2]$ equality in~\eqref{KLT} is achieved only by non-radial functions if $\mu>\mu_\bullet$ for some explicit $\mu_\bullet>\mu_\star$.
\end{theorem}
\smallskip
Notice that the definition of $\lambda_1(-\LapA-\phi)$ uses a weighted $\mathrm L^2$ norm. Using the transformations $a\mapsto a-k$, $k\in\Z$, and $a\mapsto-\,a$ that will be discussed in Section~\ref{Sec:MagneticField}, the case $a\in\R\setminus[0,1/2]$ can be reduced to the range $0\le a\le1/2$. For $a=1/2$, we shall see that $\mu_\bullet=\mu_\star=-1/4$. Further details on $\mu_\star$ and $\mu_\bullet$ will be given later. Let $\lambda\mapsto\mu(\lambda)$ be the inverse of $\mu\mapsto\lambda(\mu)$ and define 
\be{Eqn:h}
h(\lambda):=\frac p2\,(2\,\pi)^{1-\frac2p}\(\lambda+a^2\)^{\frac12+\frac1p}\(\frac{2\,\sqrt\pi\,\Gamma\big(\frac p{p-2}\big)}{(p-2)\,\Gamma\big(\frac p{p-2}+\frac12\big)}\)^{1-\frac2p}\,.
\ee
On the interval $(0,\mu_\star]$, the expression of $\lambda(\mu)$ is explicit and it will be established in Section~\ref{Sec:Thm2} that $\lambda(\mu)=h^{-1}(\mu)$ in this case while the computation of the function $h$ can be found in Appendix~\ref{Sec:Appendix}. If $a\in[0,1/2)$, the constant~$\mu_\star$ is given by $\mu_\star=h\(\lambda_\star\)$ where $\lambda_\star$ solves
\be{lambdaStar}
\(\lambda_\star+a^2\)\(p^2-4\)=4\(1-4\,a^2\)\,.
\ee

The constant $\mu_\bullet$ arises from the analysis of the \emph{symmetry breaking} phenomenon obtained by considering the linear instability of the radial optimal function. It is similar to the analysis performed by V.~Felli and M.~Schneider in~\cite{Felli-Schneider-03} for the \emph{Caffarelli-Kohn-Nirenberg inequality} without magnetic fields (see also~\cite{Catrina-Wang-01} for earlier results). The corresponding range with magnetic fields is a range of higher magnetic fields. The \emph{symmetry} range of the parameters corresponds to a weak magnetic field in which the equality case is achieved by radial potentials. The expression of $\mu_\bullet=h^{-1}(\lambda_\bullet)$ with $h$ defined by~\eqref{Eqn:h} is also explicit in terms of $\lambda_\bullet$ given below in Theorem~\ref{Thm:MgnHS}. We refer to Appendix~\ref{Sec:AppendixEv} for details.

\medskip Theorem~\ref{Thm:GroundState} has a dual counterpart which is a statement on the magnetic interpolation inequality~\eqref{Ineq:MgnHS}.
\begin{theorem}[A magnetic Hardy-Sobolev inequality]\label{Thm:MgnHS} Let $a\in[0,1/2]$ and $p>2$. For any $\lambda>-\,a^2$, there is an optimal function $\lambda\mapsto\mu(\lambda)$ which is monotone increasing and concave such that~\eqref{Ineq:MgnHS} holds for any $\psi\in\mathcal H$. With the notation of~\eqref{lambdaStar}, if $a\in (0, 1/2)$ and $\lambda\le\lambda_\star$ equality in~\eqref{Ineq:MgnHS} is achieved by
\[
\psi(x)=\(|x|^\alpha+|x|^{-\alpha}\)^{-\frac2{p-2}}\quad\forall\,x\in\R^2\,,\quad\mbox{with}\quad\alpha=\frac{p-2}2\,\sqrt{\lambda+a^2}\,.
\]
Conversely, if $a\in (0, 1/2]$ and $\lambda>\lambda_\bullet$ with
\be{Eqn:SB}
\textstyle\lambda_\bullet:=\frac{8\(\sqrt{p^4-a^2\,(p-2)^2\,(p+2)\,(3\,p-2)}+2\)-4\,p\,(p+4)}{(p-2)^3\,(p+2)}-\,a^2\,,
\ee
there is symmetry breaking, that is, the optimal functions are not radially symmetric.
\end{theorem}
\smallskip
The existence of an optimal function in~\eqref{Ineq:MgnHS} follows from a concentration-com\-pact\-ness argument as in~\cite{Catrina-Wang-01} after an Emden-Fowler transformation that will be introduced in Section~\ref{Sec:Hardy-Sobolev}. Section~\ref{Sec:Thm2} is devoted to the proof of Theorem~\ref{Thm:GroundState}. Theorem~\ref{Thm:GroundState} and Theorem~\ref{Thm:MgnHS} are equivalent as we shall see in Section~\ref{Sec:Keller-Lieb-Thirring}. The exponents $p$ and $q$ are such that $p=2\,q/(q-1)$. The value of $\lambda_\bullet$ is computed in Appendix~\ref{Sec:lambda-bullet}. Various other computational issues are dealt with in Appendix~\ref{Sec:Appendix}, as well as two figures which summarize the ranges of symmetry and symmetry breaking.

\section{Preliminary results}\label{Sec:Prelim}

\subsection{Considerations on the Aharonov-Bohm magnetic field}\label{Sec:MagneticField}

The magnetic field $B=\partial_{x_1}\mathbf A_2-\partial_{x_2}\mathbf A_1$, where $\mathbf A(x)$ is given by \eqref{eq:AB-vect}, is equal to $2\,\pi\,a\,\delta$ in the sense of distributions, where~$\delta$ denotes Dirac's distribution at $x=0$. Let us consider polar coordinates $(r,\theta)\in[0,+\infty)\times\S^1$, so that $r=|x|$, $\partial_r\psi=\nabla\psi\cdot x/r$, $\partial_\theta\psi=(-\,x_2,x_1)\cdot\nabla\psi$ and therefore $|\nabla\psi|^2=|\partial_r\psi|^2+r^{-2}\,|\partial_\theta\psi|^2$. On $\S^1\approx[0,2\,\pi)$, we consider the uniform probability measure
\[
d\sigma=(2\pi)^{-1}\,d\theta\,.
\]
We observe that for more general Aharonov-Bohm magnetic fields, $a$ could depend on $\theta$. However, by the change of gauge
\[
\psi(r,\theta)\mapsto e^{i\int_0^\theta\(a-\bar a\)\,d\theta}\,\psi(r,\theta)=:\chi(r,\theta)
\]
where $\bar a:=\iS a$ is the \emph{magnetic flux}, we notice that
\[
|\nA\psi|^2=|\partial_r\psi|^2+\frac1{r^2}\,|\(\partial_\theta-\,i\,a\)\psi|^2=|\partial_r\chi|^2+\frac1{r^2}\,|\(\partial_\theta-\,i\,\bar a\)\chi|^2\,.
\]
In this paper we shall therefore always assume that \emph{$a$ is a constant function} without loss of generality.

For any $k\in\Z$, if $\psi(r,\theta)=e^{i\,k\,\theta}\,\chi(r,\theta)$, then
\[
|\(\partial_\theta-\,i\,a\)\psi|^2=|\partial_\theta\chi+i\,(k-a)\,\chi|^2\,.
\]
Similarly, if $\chi(r,\theta)=e^{-\,i\,\theta}\,\overline\psi(r,\theta)$, we find that
\[
|\(\partial_\theta-i\,a\)\bar\psi|^2=|\partial_\theta\chi+i\,(1-a)\,\chi|^2\,,
\]
It is therefore enough to consider the case $a\in[0,1/2]$. The general case is obtained by replacing $a^2$ with $\min_{k\in\mathbb Z}\,(a-k)^2$ in all estimates.

\subsection{Magnetic kinetic energy: some estimates}\label{Sec:Kinetic}
Note that when we consider functions depending only on $\theta\in\S^1$, all integrals on $\S^1$ are computed with respect to the probability measure $d\sigma=(2\,\pi)^{-1}\,d\theta$. Our first result is inspired by~\cite[Lemma~III.2]{doi:10.1063/1.5022121}.
\begin{lemma}\label{Lem:Relaxation} Assume that $a\in[0,1/2]$ and $\psi\in C^1\cap\HA$ is such that $|\psi|>0$. Then we have
\be{Estim:Kinetic}
\ir2{|\nA\psi|^2}\ge\ir2{\(|\partial_ru|^2+\frac1{r^2}\,|\partial_\theta u|^2+\frac1{r^2}\,\frac{a^2}{\iS{u^{-2}}}\)}
\ee
where $\psi=u\,e^{iS}$. Equality holds if and only if $\partial_rS\equiv0$ and
\[
\partial_\theta S=a-\frac a{u^2}\,\frac1{\iS{u^{-2}}}\,.
\]
In the special case when $u$ does not depend on $\theta$, equality in~\eqref{Estim:Kinetic} is achieved if and only if $S$ is constant.\end{lemma}
\begin{proof} Let $S$ be such that $\psi=u\,e^{i\,S}$. We compute
\begin{eqnarray*}
|\nA\psi|^2&=&|\partial_r\psi|^2+\frac1{r^2}\,|\(\partial_\theta-\,i\,a\)\psi|^2\\
&=&|\partial_ru|^2+\frac1{r^2}\,|\partial_\theta u|^2+|u|^2\(|\partial_rS|^2+\frac1{r^2}\,|\partial_\theta S-a|^2\)\,.
\end{eqnarray*}
After dropping the term $|\partial_rS|^2$, we can optimize
\[
\ir2{\(|\partial_ru|^2+\frac1{r^2}\,|\partial_\theta u|^2+\frac1{r^2}\,|u|^2\,|\partial_\theta S-a|^2\)}
\]
over the phase~$S$ using the corresponding Euler-Lagrange equation
\[
\partial_\theta\((\partial_\theta S-a)\,u^2\)=0\,.
\]
This means that $\partial_\theta S=a+c/u^2$ for some $c=c(r)$. We integrate this identity over $\S^1$ and take into account the periodicity of $S$: for some $k\in\Z$, we have
\[
c(r)=\frac{k-a}{\iS{u^{-2}}}=\frac{k-a}{\nrmS{u^{-1}}2^2}\,.
\]
In order to minimize the magnetic kinetic energy, we have to choose the best possible $k\in\Z$ and obtain
\[
\iS{|\partial_\theta S-a|^2\,u^2}=\iS{\frac{c^2}{u^2}}=\min_{k\in\Z}\frac{(k-a)^2}{\nrmS{u^{-1}}2^2}=\frac{a^2}{\iS{u^{-2}}}\,.
\]
As a consequence, the expression of $c$ is given by $k=0$ in the equality case in~\eqref{Estim:Kinetic} and $\partial_\theta S\equiv0$ if $u$ does not depend on $\theta$. On the other hand, equality in~\eqref{Estim:Kinetic} is achieved if and only if $\partial_rS\equiv0$.\hfill\ \qed\end{proof}
\begin{lemma}\label{Lem:Kinetic} For all $a\in[0,1/2]$ and $\psi\in\mathrm H^1(\S^1)$ with $u=|\psi|$, we have
\[
\iS{|\partial_\theta\psi-\,i\,a\,\psi|^2}\ge\(1-4\,a^2\)\iS{|\partial_\theta u|^2}+a^2\iS{u^2}\,.
\]
\end{lemma}
\begin{proof}
First assume that $u=|\psi|$ is strictly positive in $\S^1$. With $\psi=u\,e^{i\,S}$, we can write
\[
\iS{|\partial_\theta\psi-\,i\,a\,\psi|^2}=\iS{\(|\partial_\theta u|^2+|\partial_\theta S-\,a|^2\,u^2\)}\,.
\]
We use the same arguments as in the proof of Lemma~\ref{Lem:Relaxation} and the inequality
\[
\nrmS{\partial_\theta u}2^2+\frac14\,\nrmS{u^{-1}}2^{-2}\ge\frac14\,\nrmS u2^2
\]
proved by P.~Exner, E.~Harrell and M.~Loss in~\cite[Section~IV]{MR1708787} to write
\begin{eqnarray*}
&&\hspace*{-24pt}\iS{|\partial_\theta\psi-\,i\,a\,\psi|^2}\\
&\ge&\(1-4\,a^2\)\iS{|\partial_\theta u|^2}+4\,a^2\(\nrmS{\partial_\theta u}2^2+\frac14\,\nrmS{u^{-1}}2^{-2}\)\\
&\ge&\(1-4\,a^2\)\iS{|\partial_\theta u|^2}+a^2\iS{u ^2}\,.
\end{eqnarray*}

Next let us consider the case when $\psi$ is equal to $0$ at some point of $\S^1$. Without loss of generality we can assume that $\psi(0)=\psi(2\,\pi)=0$ and use the diamagne\-tic inequality and Poincar\'e's inequality applied to $u=|\psi|$ in order to obtain
\begin{eqnarray*}
\iS{|(\partial_\theta-\,i\,a)\,\psi|^2}\ge\iS{|\partial_\theta u|^2}&=&\(1-4\,a^2\)\kern-3pt\iS{|\partial_\theta u|^2}+4\,a^2\kern-3pt\iS{|\partial_\theta u|^2}\\
&\ge&\(1-4\,a^2\)\kern-3pt\iS{|\partial_\theta u|^2}+a^2\kern-3pt\iS{u^2}\,.
\end{eqnarray*}
\hfill\ \qed\end{proof}

\subsection{Magnetic Hardy and non-magnetic Hardy-Sobolev inequalities}\label{Sec:Hardy-Sobolev}~

In dimension $d=2$, the magnetic Hardy inequality~\eqref{LaWe} holds, as was proved in~\cite{MR1708811}. When $a\in[0,1/2]$, we notice that \hbox{$\displaystyle\min_{k\in\mathbb Z}\,(a-k)^2=a^2$}.

The weighted interpolation inequality
\be{CKN}
\ir2{\frac{|\nabla v|^2}{|x|^{2\CKNa}}}\ge\mathsf C_\CKNa\(\ir2{\frac{|v|^p}{|x|^{\CKNb\,p}}}\)^{2/p}\quad\forall\,v\in\mathcal D(\R^2)
\ee
is known in the literature as the \emph{Caffarelli-Kohn-Nirenberg inequality} according to~\cite{Caffarelli-Kohn-Nirenberg-84} but was apparently discovered earlier by V.P.~Il'in, see~\cite{Ilyin}. The exponent $\CKNb=\CKNa+2/p$ is determined by the scaling invariance, the inequality can be extended by density to a space larger than the space $\mathcal D(\R^2)$ of smooth functions with compact support, and as $p$ varies in $(2,\infty)$, the parameters $\CKNa$ and $\CKNb$ are such that
\[
\CKNa<\CKNb\le\CKNa+1\quad\mbox{and}\quad\CKNa<0\,.
\]
Moreover, it is also possible to consider the case $\CKNa>0$ in an appropriate functional space after a Kelvin-type transformation: see~\cite{Catrina-Wang-01,DELT09}, but we will not consider this case here. As noticed for instance in~\cite{DELT09}, by considering $v(x)=|x|^\CKNa\,u(x)$, Ineq.~\eqref{CKN} is equivalent to the \emph{Hardy-Sobolev inequality}
\be{CKN2}
\ir2{|\nabla u|^2}+\CKNa^2\ir2{\frac{|u|^2}{|x|^2}}\ge\mathsf C_\CKNa\(\ir2{\frac{|u|^p}{|x|^2}}\)^{2/p}\quad\forall\,u\in\mathcal D(\R^2)\,.
\ee
By linear instability, see~\cite{Felli-Schneider-03}, the optimal functions for~\eqref{CKN} are not radially symmetric if $\CKNb<\CKNb_{\rm FS}(\CKNa):=\CKNa-\CKNa/\sqrt{1+\CKNa^2}$. The main ingredient of the proof is reproduced in Appendix~\ref{Sec:AppendixEv}, in the two-dimensional case. On the contrary, if $\CKNb_{\rm FS}(\CKNa)\le\CKNb\le\CKNa+1$, we learn from~\cite{MR3570296} that equality in~\eqref{CKN} is achieved by
\begin{equation}\label{vstar}
v_\star(x)=\(1+|x|^{-(p-2)\,\CKNa}\)^{-\frac2{p-2}}\quad\forall\,x\in\R^d
\end{equation}
up to a scaling and a multiplication by a constant. In the range $\CKNb_{\rm FS}(\CKNa)\le\CKNb\le\CKNa+1$, this provides us with the value of the optimal constant, namely $\mathsf C_\CKNa=\mathsf C_\CKNa^\star$ where the expression of $\mathsf C_\CKNa^\star$ is given in Appendix \ref{Sec:AppendixOptCst}. We observe that for any given $\CKNb\in(0,1)$, we have that $\lim_{\CKNa\to0_-}\mathsf C_\CKNa^\star=+\infty$, so that the inequality does not make sense for $\CKNa=0$. Using polar coordinates~$(r,\theta)$, the Emden-Fowler transformation
\be{EF}
u(r,\theta)=w(s,\theta)\,,\quad s=-\log r
\ee
turns Ineq.~\eqref{CKN2} into
\be{CKN3}
\iC{\(w_s^2+w_\theta^2+\CKNa^2\,w^2\)}\ge\mathsf K_\CKNa\(\iC{|w|^p}\)^{2/p}\quad\kern-5pt\forall\,w\in\mathrm H^1(\R\times\S^1)
\ee
with $\mathsf K_\CKNa:=(2\,\pi)^{\frac2p-1}\,\mathsf C_\CKNa$. We refer to~\cite{1703} for a more detailed review on the Caffarelli-Kohn-Nirenberg inequality. For any given $p>2$, we define
\be{CK}
\mathsf K_\CKNa^\star:=(2\,\pi)^{\frac2p-1}\,\mathsf C_\CKNa^\star\quad\mbox{and}\quad\mathsf k^\star(\lambda)=\mathsf K_{\sqrt{a^2+\lambda}}^\star
\ee
so that $\mathsf K_\CKNa^\star$ is the optimal constant in~\eqref{CKN3} restricted to \emph{symmetric functions}, that is, functions depending only on $s$. See Appendix~\ref{Sec:AppendixOptCst} for the explicit expression of~$\mathsf K_\CKNa^\star$. For our purpose, we have to consider a slightly more general problem. For any $w\in\mathrm H^1(\R\times\S^1)$, let us define
\be{Fkappanu}
\mathcal F_{\kappa,\nu}[w]:=\iC{\(w_s^2+\nu\,w_\theta^2+\kappa\,w^2\)}-\mathsf K_{\sqrt\kappa}^\star\(\iC{|w|^p}\)^{2/p}\,.
\ee
\begin{lemma}\label{Lem:Sym} Let $p>2$, $\kappa>0$ and $\nu>0$. Then $\mathcal F_{\kappa,\nu}$ has a minimizer $w\in\mathrm H^1(\R\times\S^1)$ such that $\|w\|_{\mathrm L^p(\R\times\S^1)}=1$ and $w$ depends only on $s$ if and only if
\[
\kappa\,(p^2-4)\le4\,\nu\,.
\]\end{lemma}
\begin{proof} The existence of a minimizer is obtained as in the standard case corresponding to $\nu=1$ and we refer to~\cite{MR2966111} for the details. The function $w_\star(s,\theta)=v_\star\(e^{-s},\theta\)$, where $v_\star$ is defined by \eqref{vstar}, is a critical point of $\mathcal F_{\kappa,\nu}$ such that $\mathcal F_{\kappa,\nu}[w_\star]=0$ and $w_\star$ is linearly instable if and only if $\kappa\,(p^2-4)>4\,\nu$ (see Appendix~\ref{Sec:AppendixEv}). By adapting~\cite[Corollary~1.3]{MR3570296}, the minimizer of $\mathcal F_{\kappa,\nu}$ is independent of the angular variable $\theta$ if and only if $\kappa\,(p^2-4)\le4\,\nu$. In that case, we have $w(s)=w_\star(s-s_0)$ for some $s_0\in\R$.\hfill\ \qed\end{proof}

\section{Proofs}\label{Sec:Proofs}

\subsection{Magnetic interpolation inequalities}\label{Sec:Thm2}

We prove Theorem~\ref{Thm:MgnHS}. For more readability, we split the proof into three steps.

\medskip\noindent\emph{Step 1 -- Ineq.~\eqref{Ineq:MgnHS} without the optimal constant}. Let $t\in[0,1]$. From the diamagnetic inequality, we get
\[
\nrmdeux{\nA\psi}2\ge\nrmdeux{\nabla u}2
\]
where $u=|\psi|$, and therefore,
\begin{eqnarray*}
\ir2{|\nA\psi|^2}+\lambda\ir2{\frac{|\psi|^2}{|x|^2}}&\ge&t\(\nrmdeux{\nA\psi}2^2-\,a^2\ir2{\frac{u^2}{|x|^2}}\)\\
&&+\,(1-t)\(\nrmdeux{\nabla u}2^2+\frac{\lambda+a^2\,t}{1-t}\kern-2pt\ir2{\frac{u^2}{|x|^2}}\).
\end{eqnarray*}
Using~\eqref{LaWe} and~\eqref{CKN2} applied with $\CKNa^2=\frac{\lambda+a^2\,t}{1-t}$, $t\in(0,1)$ such that $\lambda+a^2\,t>0$, this estimate proves the existence of a positive constant $\mu(\lambda)$ in~\eqref{Ineq:MgnHS}. As an infimum on $\mathcal H$ of affine non-decreasing functions of $\lambda$, the function $\lambda\mapsto\mu(\lambda)$ is concave and non-decreasing.

\medskip\noindent\emph{Step 2 -- Optimal estimate in the symmetry range}. With $a\in[0,1/2]$, $\psi\in\mathcal H(\R^2)$ and $u=|\psi|$, we know from Lemma~\ref{Lem:Kinetic} that
\[
\ir2{|\nA\psi|^2}\ge\ir2{|\partial_r u|^2}+\(1-4\,a^2\)\ir2{\frac1{r^2}\,|\partial_\theta u|^2}+a^2\ir2{\frac1{r^2}\,u^2}\,.
\]
We can estimate the optimal constant $\mu(\lambda)$ in~\eqref{Ineq:MgnHS} by the optimal constant $\mu_{\mathrm{rel}}(\lambda)$ in the relaxed inequality
\[
\ir2{\(|\partial_ru|^2+{\textstyle\frac{1-\,4\,a^2}{r^2}}\,|\partial_\theta u|^2\)}\,+\(\lambda+a^2\)\ir2{\frac{|u|^2}{|x|^2}}\ge\mu_{\mathrm{rel}}(\lambda)\(\ir2{\frac{|u|^p}{|x|^2}}\)^\frac2p.
\]
Using the Emden-Fowler transformation~\eqref{EF}, this inequality can be rewritten on the cylinder $\R\times\S^1$ as
\begin{eqnarray*}
\iC{\(|\partial_sw|^2+\(1-4\,a^2\)\,|\partial_\theta w|^2\)}+\(\lambda+a^2\)\iC{|w|^2}\hspace*{1.4cm}&&\\
\ge(2\,\pi)^{\frac2p-1}\,\mu_{\mathrm{rel}}(\lambda)\(\iC{|w|^p}\)^\frac2p.&&
\end{eqnarray*}
By Lemma~\ref{Lem:Sym} applied with $\nu=1-4\,a^2$ and $\kappa=\lambda+a^2$, the optimal function for the above inequality is independent of the angular variable $\theta$ if and only if
\[
\(\lambda+a^2\)\(p^2-4\)\le4\(1-4\,a^2\)\,,
\]
that is, $\lambda\le\lambda_\star$ with $\lambda_\star$ defined by the equality case, \emph{i.e.}, by~\eqref{lambdaStar}. If $a=1/2$, note that there is no $\lambda$ such that $-\,a^2<\lambda\le\lambda_\star$. With $a\in(0,1/2)$ and $\lambda_\star$ defined by~\eqref{lambdaStar}, this amounts to $\lambda\le\lambda_\star$. In that case the optimal function~is
\[
w_\star(s):=\(\cosh\(\frac{p-2}2\,\sqrt{\lambda+a^2}\,s\)\)^{-\frac2{p-2}}\quad\forall\,s\in\R
\]
up to a multiplication by a constant and a translation (in the $s$ variable). This determines the value of $\mu_{\mathrm{rel}}(\lambda)$. By construction, we know that $\mu(\lambda)\ge\mu_{\mathrm{rel}}(\lambda)$, but using $(r,\theta)\mapsto w_\star(-\log r)$ as a test function  in~\eqref{Ineq:MgnHS}, we find that $\mu(\lambda)=\mu_{\mathrm{rel}}(\lambda)$ if $\lambda\le\lambda_\star$. See Appendix~\ref{Sec:AppendixEv} for details on the computation of $\lambda_\star$.

\medskip\noindent\emph{Step 3 -- The symmetry breaking range}. This range is the set of $\lambda$ and $a$ for which the optimal functions are not symmetric functions. Let 
\[
\mathcal E_{a,\lambda}[\psi]:=\ir2{|\nA\psi|^2}+\lambda\ir2{\frac{|\psi|^2}{|x|^2}}-\mu(\lambda)\(\ir2{\frac{|\psi|^p}{|x|^2}}\)^{2/p}.
\]
We produce a direction of instability for $\mathcal E_{a,\lambda}$ by perturbing the phase and the modulus of $\psi(x)=w_\star(-\log|x|)$ simultaneously. Let us start by some preliminary computations. Define $c_\omega(s)=\cosh(\omega\,s)$, $s_\omega(s)=\sinh(\omega\,s)$ and $\mathsf I_\alpha:=\is{c_\omega^{-\alpha}}$. An integration by parts shows that
\begin{eqnarray*}
\mathsf J_{\alpha+2}:=\is{s_\omega^2\,c_\omega^{-(\alpha+2)}}&=&-\frac1{\alpha+1}\,\frac 1\omega\is{s_\omega\(c_\omega^{-(\alpha+1)}\)'}\\
&=&\frac1{\alpha+1}\is{c_\omega^{-\alpha}}=\frac{\mathsf I_\alpha}{\alpha+1}\,.
\end{eqnarray*}
On the other hand, using the identity $c_\omega^2-s_\omega^2=1$, we obtain that
\[
\mathsf I_{\alpha+2}=\is{\(c_\omega^2-s_\omega^2\)c_\omega^{-(\alpha+2)}}=\(1-\frac1{\alpha+1}\)\mathsf I_\alpha=\frac\alpha{\alpha+1}\,\mathsf I_\alpha\,.
\]
With the choice $\mu=\(2\pi\iC{|w_\star|^p}\)^{1-2/p}$ corresponding to the optimal constant achieved by the symmetric function $w_\star$, with $s=-\,\log r$, by considering $\psi_\varepsilon(r,\theta):=\big(w_\star(s)+\varepsilon\,\varphi(s,\theta)\big)\,\exp\big(i\,\varepsilon\,\chi(s,\theta)\big)$, at order $\varepsilon^2$ we obtain that $\mathcal E_{a,\lambda}[\psi_\varepsilon]=\varepsilon^2\,\mathcal Q[\varphi,\chi]+o(\varepsilon^2)$ where $\mathcal Q$ is the quadratic form defined by
\begin{eqnarray*}
\mathcal Q[\varphi,\chi]
&=&\iC{w_\star^2\(|\partial_s\chi|^2+|\partial_\theta\chi-\,a|^2-a^2\)}-4\,a\iC{w_\star\,\varphi\,\partial_\theta\chi}\\
&&+\iC{\(|\partial_s\varphi|^2+|\partial_\theta\varphi|^2+\(\lambda+a^2\)\varphi^2\)}\\
&&-\,(p-1)\iC{|w_\star|^{p-2}\,|\varphi|^2}
\end{eqnarray*}
and
\[
w_\star(s)=\zeta_\star\,\big(c_\omega(s)\big)^{-\frac2{p-2}}\quad\mbox{with}\;\zeta_\star=\(\frac p2\(\lambda+a^2\)\)^\frac1{p-2}\;\mbox{and}\;\omega=\frac{p-2}2\,\sqrt{\lambda+a^2}\,.
\]
With the \emph{ansatz}
\be{Ansatz-phi-chi}
\varphi(s,\theta)=\big(c_\omega(s)\big)^{-\frac p{p-2}}\,\cos\theta\quad\mbox{and}\quad\chi(s,\theta)=\frac\zeta{\zeta_\star}\,\big(c_\omega(s)\big)^{-1}\sin\theta
\ee
where $\zeta$ is a parameter to be fixed later, we obtain that
\begin{eqnarray*}
\mathcal Q[\varphi,\chi]&=&\zeta^2\(\omega^2\,\mathsf J_{\alpha+2}+\mathsf I_\alpha\)-4\,\zeta\,a\,\mathsf I_\alpha\\
&&+\(\frac{p\,\omega}{p-2}\)^2\,\mathsf J_{\alpha+2}+\(1+\lambda+a^2\)\mathsf I_\alpha-(p-1)\,\zeta_\star^{p-2}\,\mathsf I_{\alpha+2}
\end{eqnarray*}
with $\alpha=2\,p/(p-2)$. We minimize the expression of $\mathcal Q[\varphi,\chi]$ with respect to $\zeta\in\R$, that is, we take
\[
\zeta=\frac{2\,a\,\mathsf I_\alpha}{\omega^2\,\mathsf J_{\alpha+2}+\mathsf I_\alpha}\,.
\]
After replacing $\alpha$, $\zeta$, $\zeta_\star$, and $\omega$ by their values in terms of $a$, $p$ and $\lambda$, we find that the infimum of the admissible parameters $\lambda>-\,a^2$ for which $\mathcal Q[\varphi,\chi]<0$ is given by~\eqref{Eqn:SB}. Hence we know that there is symmetry breaking for any $\lambda>\lambda_\bullet$. This concludes the proof of Theorem~\ref{Thm:MgnHS}.\hfill\ \qed

\begin{remark} The function $\psi_\varepsilon$ used in Step 3 of the proof of Theorem~\ref{Thm:MgnHS} to produce a negative direction of variation of $\mathcal E_{a,\lambda}$ is only a test function which couples the modulus and the phase. To get an optimal range with this method, one should identify the lowest eigenvalue in the system associated with the variation of $\mathcal Q$: see Section~\ref{Sec:System} for details. This is so far an open question as the corresponding eigenfunctions are not identified yet.
 
One may wonder if a better result could be achieved by varying only the modulus using the function $\psi_\varepsilon(r,\theta):=w_\star(-\,\log r)+\varepsilon\,\varphi(-\,\log r,\theta)$ and choosing the optimal $\varphi$. In that case, the instability is reduced to the instability of $\mathcal F_{\kappa,\nu}$ as defined by~\eqref{Fkappanu} with $\kappa=\lambda+a^2$ and $\nu=1$, which is the classical computation of~\cite{Felli-Schneider-03} (also see Appendix~\ref{Sec:AppendixEv}): here instability occurs if
\be{lambdaFS}
\lambda>\lambda_{\mathrm{FS}}(a):=\frac{4\(1+\,a^2\)-\,a^2\,p^2}{p^2-4}\,.
\ee
Elementary considerations show that $\lambda_\bullet<\lambda_{\mathrm{FS}}$ and that the threshold given by~$\lambda_\bullet$ is by far better (see Fig.~\ref{Fig1}).\end{remark}

\subsection{Spectral estimates}\label{Sec:Keller-Lieb-Thirring}

The proof of Theorem~\ref{Thm:GroundState} is a simple consequence of the estimate
\[
\EA[\psi]=\ir2{\(|\nA\psi|^2-\phi\,|\psi|^2\)}\ge\ir2{|\nA\psi|^2}-\,\vertiii\phi_q\,\(\ir2{\frac{|\psi|^p}{|x|^2}}\)^{2/p}
\]
by H\"older's inequality, with $q=p/(p-2)$. With the notation $\mu=\vertiii\phi_q$, Eq.~\eqref{KLT} is a consequence of Eq.~\eqref{Ineq:MgnHS} in Theorem~\ref{Thm:MgnHS}: the right-hand side is bounded from below by $-\,\lambda(\mu)\,\nrm{|x|^{-1}\,\psi}2^2$.

Reciprocally Theorem~\ref{Thm:MgnHS} follows from Theorem~\ref{Thm:GroundState} with $\phi=|x|^{-2}\,|\psi|^{p-2}$, which corresponds to the equality case in the above H\"older inequality.\hfill\ \qed

\appendix\section{Appendix}\label{Sec:Appendix}

\subsection{Optimal constants in the symmetric case}\label{Sec:AppendixOptCst}

It is known from~\cite{zbMATH02502560} that
\[\label{Kstar}
\mathsf K_\CKNa^\star=\frac p2\,|\CKNa|^{1+\frac2p}\(\frac{2\,\sqrt\pi\,\Gamma\big(\frac p{p-2}\big)}{(p-2)\,\Gamma\big(\frac p{p-2}+\frac12\big)}\)^{1-\frac2p}
\]
is the optimal constant in the inequality
\[
\is{|w'|^2}+\CKNa^2\is{|w|^2}\ge\mathsf K_\CKNa^\star\(\is{|w|^p}\)^{2/p}\quad\forall\,w\in\mathrm H^1(\R)\,.
\]
For any given $p>2$, the optimal constant in~\eqref{CKN3} is also $\mathsf K_\CKNa^\star$ in the particular case of symmetric functions, because $d\sigma$ is the uniform probability measure on~$\S^1$. See for instance~\cite{delPino20102045} for details. 

The optimal constants in~\eqref{CKN2} and~\eqref{CKN3} are related by~\eqref{CK}. In the symmetry range, we have $\mathsf K_\CKNa=\mathsf K_\CKNa^\star=(2\,\pi)^{\frac2p-1}\,\mathsf C_\CKNa^\star=(2\,\pi)^{\frac2p-1}\,\mathsf C_\CKNa$, where
\[
\mathsf C_\CKNa^\star=\frac p2\,(2\,\pi)^{1-\frac2p}\,|\CKNa|^{1+\frac2p}\(\frac{2\,\sqrt\pi\,\Gamma\big(\frac p{p-2}\big)}{(p-2)\,\Gamma\big(\frac p{p-2}+\frac12\big)}\)^{1-\frac2p}\,.
\]
This expression can be recovered by writing that the equality case in~\eqref{CKN2} is achieved by the function $v_\star$. Indeed, with the change of variables $(r,\theta)\mapsto(s,\theta)$ with $s=r^{-\CKNa\,(p-2)/2}$ and $n=2\,p/(p-2)$ as in~\cite[Section~3.1]{MR3570296}, and $f_\star(s)=v_\star(r)$, this means that
\[
\mathsf C_\CKNa^\star=(2\,\pi)^{1-\frac2p}\(\frac\CKNa2\,(p-2)\)^{1+\frac2p}\,\frac{\int_0^{+\infty}|f_\star'|^2\,s^{n-1}\,ds}{\(\int_0^{+\infty}|f_\star|^p\,s^{n-1}\,ds\)^{2/p}}
\]
where $f_\star$ is the Aubin-Talenti function $f_\star(s)=\(1+s^2\)^{-(n-2)/2}$.

\subsection{Ground state eigenvalues of the quadratic form}\label{Sec:AppendixEv}

\subsubsection*{$\bullet$ Linearization and eigenvalues: the one-dimensional case}
Let us summarize some classical results on the linearization of the Gagliardo-Nirenberg inequalities in the one-dimensional case, based on~ \cite[Appendix~A.2]{0951-7715-27-3-435}. According, \emph{e.g.}, to~\cite{Dolbeault06082014}, the function $\overline w(s)=(\cosh s)^{-\frac 2{p-2}}$ is the unique positive solution of
\[
-\,(p-2)^2\,\overline w''+4\,\overline w-2\,p\,\overline w^{p-1}=0
\]
on $\R$, up to translations. The function $w(s):=\alpha\,\overline w(\beta\,s)$ solves
\[
-\,w''+\frac{4\,\beta^2}{(p-2)^2}\,w-\frac{2\,p\,\beta^2}{(p-2)^2}\,\alpha^{2-p}\,w^{p-1}=0\,.
\]
With $\beta=\frac{p-2}2\,\sqrt\kappa$ and $\alpha=(\frac p2\,\kappa)^\frac1{p-2}$, $w=\alpha\,w_\star$ is given by
\[
w(s)=\(\frac p2\,\kappa\)^\frac1{p-2}\left[\cosh\(\frac{p-2}2\,\sqrt\kappa\,s\)\right]^{-\frac 2{p-2}}\quad\forall\,s\in\R
\]
and solves
\[
-\,w''+\kappa\,w=|w|^{p-2}\,w\,.
\]

The ground state energy $\lambda_1(\mathcal H_\kappa)$ of the P\"oschl-Teller operator
\[
\mathcal H_\kappa:=-\,\frac{d^2}{ds^2}+\,\kappa-\,(p-1)\,w^{p-2}
\]
is characterized as follows. The function
\[
\varphi_1(s):=\alpha^\frac p2\,\big(\cosh(\beta\,s)\big)^{-\frac p{p-2}}=w^\frac p2
\]
solves
\[
-\,\varphi_1''+\frac14\,\kappa\,p^2\,\varphi_1-\,(p-1)\,w^{p-2}\,\varphi_1=0
\]
and therefore provides the principal eigenvalue of $\mathcal H_\kappa$,
\[
\lambda_1(\mathcal H_\kappa)=-\,\frac\kappa4\(p^2-\,4\)\,.
\]
The Sturm-Liouville theory guarantees that $\varphi_1$ generates the ground state.

\subsubsection*{$\bullet$ The lowest non-radial mode on the cylinder}
Let us consider the operator
\[
\mathcal H_{\kappa,\nu}:=-\,\frac{\partial^2}{\partial s^2}-\,\nu\,\frac{\partial^2}{\partial\theta^2}+\,\kappa-\,(p-1)\,w^{p-2}
\]
on the cylinder $\R\times\S^1\ni(s,\theta)$. By separation of variables, the lowest non-symmetric eigenvalue is associated with the function $\varphi(s,\theta)=\varphi_1(s)\,\cos\theta$, so that the ground state of $\mathcal H_{\kappa,\nu}$ is
\[
\lambda_1(\mathcal H_{\kappa,\nu})=\nu-\,\frac\kappa4\(p^2-\,4\)\,.
\]

\subsubsection*{$\bullet$ Lowest eigenvalues and threshold for the linear instability}
Whenever the optimal function in~\eqref{Ineq:MgnHS} is radially symmetric, we get that $\mu(\lambda)=\mathsf C_\CKNa^\star$ with $\CKNa=\sqrt{a^2+\lambda}$, \emph{i.e.}, $\mu(\lambda)=(2\,\pi)^{1-\frac2p}\,\mathsf k^\star(\lambda)$, or
\be{mu-lambda}
\mu(\lambda)=\frac p2\,(2\,\pi)^{1-\frac2p}\(\lambda+a^2\)^{\frac12+\frac1p}\(\frac{2\,\sqrt\pi\,\Gamma\big(\frac p{p-2}\big)}{(p-2)\,\Gamma\big(\frac p{p-2}+\frac12\big)}\)^{1-\frac2p}\,.
\ee
With $\mathcal F_{\kappa,\nu}$ defined by~\eqref{Fkappanu}, let us consider a Taylor expansion of $\mathcal F_{\kappa,\nu}[w_\varepsilon]$ with $w_\varepsilon(s,\theta):=w_\star(s)+\varepsilon\,\varphi(s,\theta)$ at order two with respect to $\varepsilon$. For $\varepsilon>0$ small enough, the sign of $\mathcal F_{\kappa,\nu}[w_\varepsilon]$ is determined by the sign of the quadratic form
 \begin{eqnarray*}
\varphi\mapsto\iC{\(|\partial_s\varphi|^2+\nu\,|\partial_\theta\varphi|^2+\kappa\,\varphi^2\)}\hspace*{4cm}&&\\
-\,(2\,\pi)^{\frac2p-1}\,(p-1)\,\mu\(\iC{w_\star^p}\)^{\frac2p-1}\iC{w_\star^{p-2}\,|\varphi|^2}\,.&&
\end{eqnarray*}
Hence $\mathcal F_{\kappa,\nu}$ can be made negative by choosing $\varphi=\varphi_1$, which shows that $w_\star$ is an instable critical point of $\mathcal F_{\kappa,\nu}$ if and only if $\lambda_1(\mathcal H_{\kappa,\nu})<0$. Notice that Lemma~\ref{Lem:Sym} states the reverse result, which is the difficult part of the result: whenever $\lambda_1(\mathcal H_{\kappa,\nu})\ge0$, the minimum of $\mathcal F_{\kappa,\nu}$ is achieved by $w_\star$ so that $\mathcal F_{\kappa,\nu}\ge\mathcal F_{\kappa,\nu}[w_\star]\ge0$.

Applied with $\kappa=\lambda+a^2$ and $\nu=1$, we recover the computation of~\cite{Felli-Schneider-03}, which determines $\lambda_{\mathrm{FS}}$ as in~\eqref{lambdaFS}. Applied with $\kappa=\lambda+a^2$ and $\nu=1-\,4\,a^2$, we obtain that
\[
\lambda_1(\mathcal H_{\lambda+a^2,1})-\,4\,a^2=1-\,\frac14\(\lambda+a^2\)\(p^2-\,4\)-\,4\,a^2
\]
and observe that it is negative if and only if
\[
\lambda_\star(a)=\frac{4\(1-3\,a^2\)-\,a^2\,p^2}{p^2-4}\,.
\]
Using $\mu(\lambda)=(2\,\pi)^{1-\frac2p}\,\mathsf k^\star(\lambda)$ in the symmetry range, we obtain that $\mu_\star(a)=\mu\big(\lambda_\star(a)\big)$ with $\mu$ given by~\eqref{mu-lambda}, \emph{i.e.},
\[
\mu_\star(a)=2\,p\(\frac{1-4\,a^2}{p^2-4}\)^{\frac12+\frac1p}\pi^{\frac32-\frac3p}
\(\frac{2\,\Gamma\big(\frac p{p-2}\big)}{(p-2)\,\Gamma\big(\frac p{p-2}+\frac12\big)}\)^{1-\frac2p}\,.
\]

\subsection{Computation of \texorpdfstring{$\lambda_\bullet$}{lambda-bullet}}\label{Sec:lambda-bullet}

Let us give some details on the computation of $\lambda_\bullet$. An expansion of $\mathcal Q[\varphi,\chi]$ as defined in Section~\ref{Sec:Thm2} computed with the ansatz~\eqref{Ansatz-phi-chi} shows that it has the sign of
\[
\textstyle\mathsf q(\lambda):=-\,\lambda^2-\,2\(4\,\frac{p^2+4\,p-4}{(p-2)^3\,(p+2)}+\,a^2\)\lambda+\,8\,\frac{2\,(3\,p-2)-a^2\(p^3+2\,p^2+12\,p-8\)}{(p-2)^3\,(p+2)}-a^4\,.
\]
Since $\mathsf q(\lambda_\star)=\big(\frac{8\,a}{p^2-4}\big)^2\,(1-4\,a^2)$ is positive for any $a\in(0,1/2)$ and since $\lim_{\lambda\to\infty}\mathsf q(\lambda)=-\infty$, we know that~$\lambda_\bullet$ defined by $\mathsf q(\lambda_\bullet)=0$ is such that $\lambda_\bullet>\lambda_\star$. Notice that the other root of $\mathsf q(\lambda)=0$ is in the range $(-\infty,-a^2)$, and that the discriminant $p^4-a^2\,(p-2)^2\,(p+2)\,(3\,p-2)$ is positive for any $(a,p)\in(0,1/2)\times(2,+\infty)$. Additionally, we obtain by direct computation that
\[
\textstyle\lambda_\bullet-\lambda_\star=\frac8{(p-2)^3\,(p+2)}\(\sqrt{p^4-a^2\,(p-2)^2\,(3\,p^2+4\,p-4)}+2\,a^2\,(p-2)^2-p^2\)
\]
is positive for any $a\in(0,1/2)$. Numerically this difference turns out to be very small: see Figs.~\ref{Fig1} and~\ref{Fig2}.

\begin{figure}[ht]
\begin{center}
\includegraphics[width=6cm,height=3.5cm]{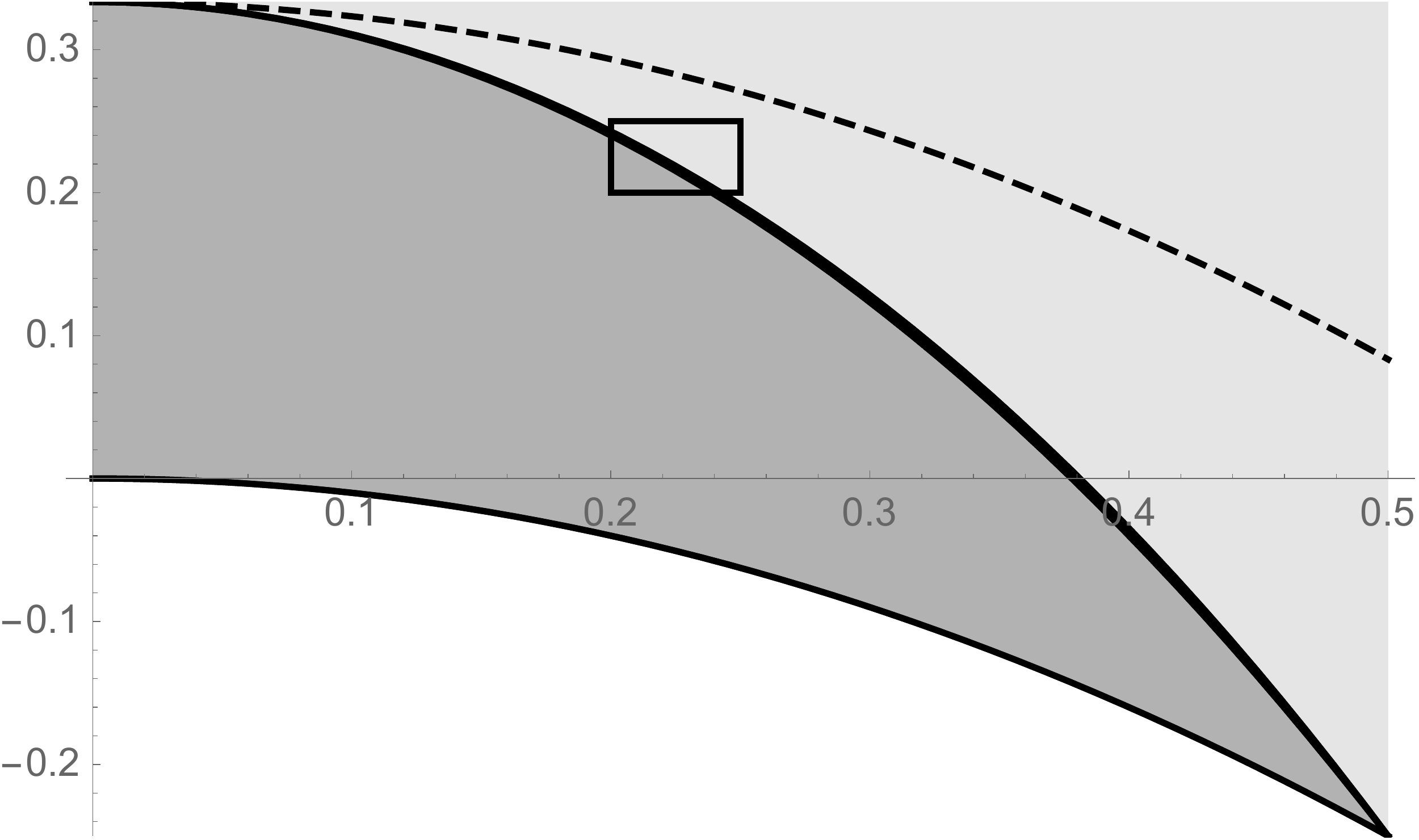}\hspace*{6pt}\includegraphics[width=6cm,height=3.5cm]{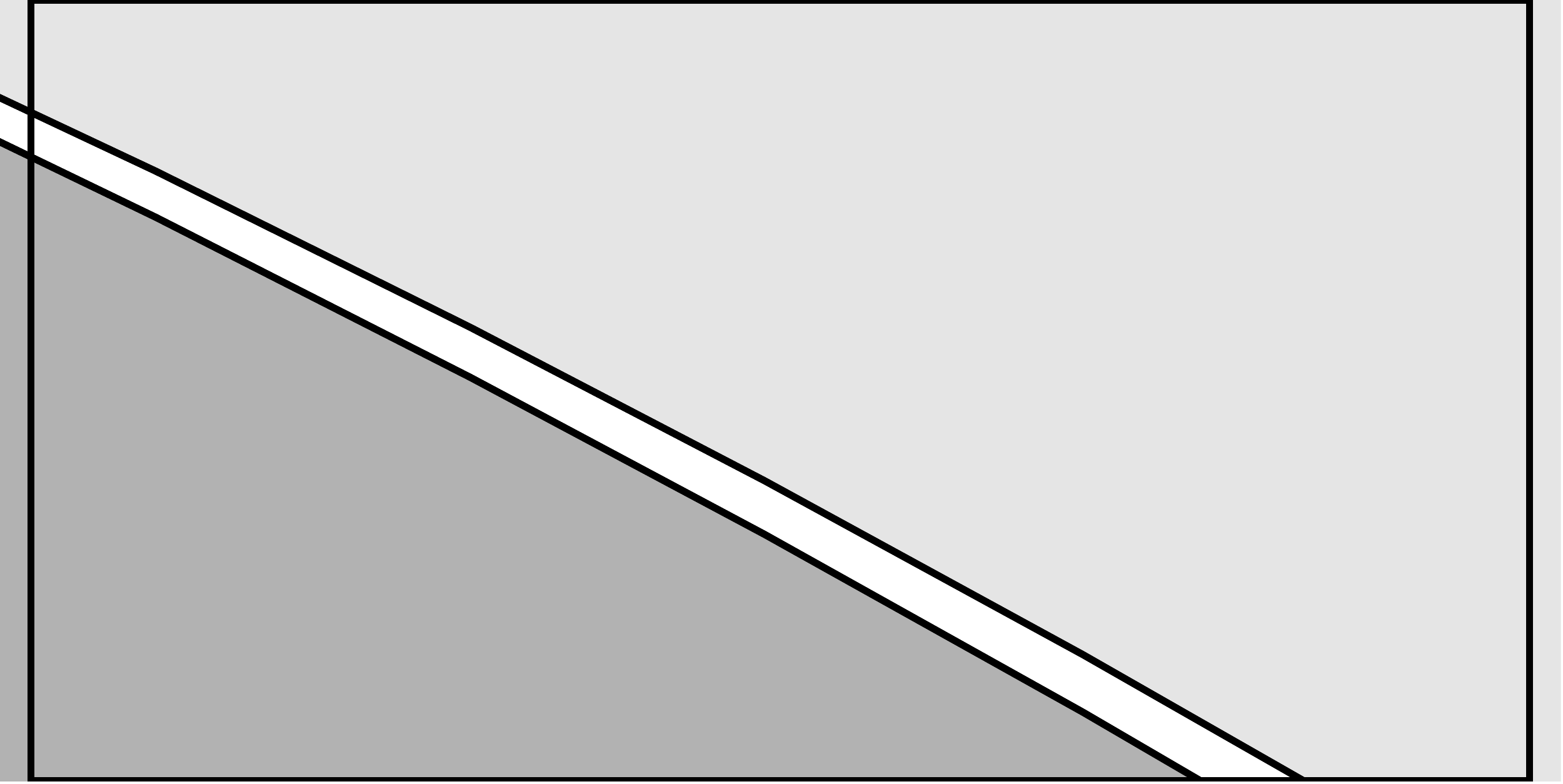}
\caption{\label{Fig1} Here we assume that $a\in(0,1/2)$ and consider the case $p=4$. Left: The region of symmetry is the dark grey area which lies between the curves $a\mapsto-\,a^2$ and $a\mapsto\lambda_\star(a)$. The light grey area above $a\mapsto\lambda_\bullet(a)$ is the region of symmetry breaking. The curve $a\mapsto\lambda_{\mathrm{FS}}(a)$ is the dashed curve, above which the symmetry breaking is shown by considering only a perturbation of the modulus. It is a poor estimate away from a neighborhood of $a=0$. Right: An enlargement of the boxed area shows that $\lambda_\bullet$ and $\lambda_\star$ do not coincide. Also see Fig.~\ref{Fig2}.}
\end{center}
\end{figure}

\begin{figure}[ht]
\begin{center}
\includegraphics[width=6cm,height=4cm]{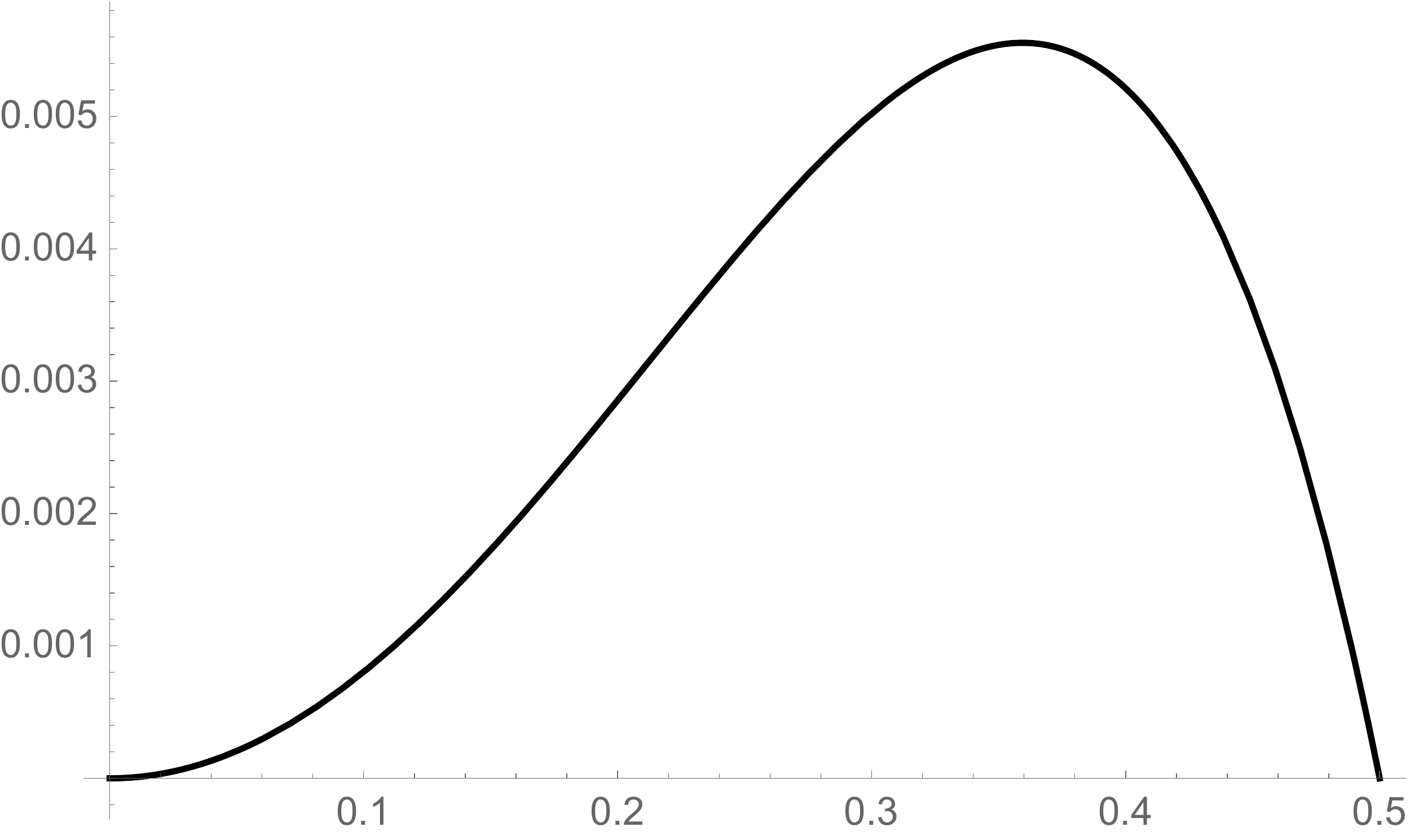}
\caption{\label{Fig2} The curve $a\mapsto\lambda_\bullet(a)-\lambda_\star(a)$ with $p=4$ shows that there is a little gap between the symmetry and the symmetry breaking region, which is to be expected because $\lambda_\star$ is determined by a non-optimal test function.}
\end{center}
\end{figure}

\subsection{The range of linear instability of the magnetic interpolation inequality}\label{Sec:System}

The high level of accuracy shown in Fig.~\ref{Fig2} deserves some comments. For any given $(a,p)\in(0,1/2)\times(2,+\infty)$, the threshold between the symmetry region and the symmetry breaking region is in the interval $(\lambda_\star,\lambda_\bullet)$. Since $\lambda_\bullet$ is determined by the choice of~\eqref{Ansatz-phi-chi}, one has to understand why this test function gives such a precise estimate.

Let us consider an \emph{ansatz} in which only the angular dependence is fixed. With
\[
\varphi(s,\theta)=H(z)\,\cos\theta\quad\mbox{and}\quad\chi(s,\theta)=G(z)\,\frac{\sin\theta}{w_\star(s)}
\]
and the change of variables
\[
z=\tanh\(\omega\,s\)\quad\mbox{and}\quad\omega=\frac{p-2}2\,\sqrt{\lambda+a^2}\,,
\]
the computation of $\mathcal Q[\varphi,\chi]$ is reduced to the computation of
\begin{eqnarray*}\textstyle
&&\hspace*{-18pt}\mathsf Q(G,H)\\
&:=&\textstyle\iz{\(\omega^2\(1-z^2\)\,|G'|^2+\(1+\frac{4\,\omega^2}{(p-2)^2}\)\frac{G^2}{1-z^2}-\frac{2\,p\,\omega^2}{(p-2)^2}\,G^2-\, \frac{4\,a}{1-z^2}\,G\,H\)}\\
&&\textstyle+\iz{\(\omega^2\(1-z^2\)\,|H'|^2+\(1+\frac{4\,\omega^2}{(p-2)^2}\)\frac{H^2}{1-z^2}-(p-1)\,\frac{2\,p\,\omega^2}{(p-2)^2}\,H^2\)}
\end{eqnarray*}
using
\[
\frac{dz}{ds}=\omega\(1-z^2\)\,,\quad\lambda+a^2=\frac{4\,\omega^2}{(p-2)^2}\quad\mbox{and}\quad w_\star^{p-2}=\frac{2\,p\,\omega^2}{(p-2)^2}\(1-z^2\)\,.
\]
We recover the expression of $\mathsf q(\lambda)$ with the choice
\[
G[z]=\zeta\,H[z]\quad\mbox{and}\quad H(z)=\(1-z^2\)^{\frac p{2\,(p-2)}}\,,
\]
after optimizing on $\zeta$. All computations done, the optimal value of $\zeta$ is
\[
\zeta=\frac{a\,(p+2)\,(3\,p-2)}{p^2+\sqrt{p^4-a^2\,(p-2)^2
\,(p+2)\,(3\,p-2)}}\,,
\]
and we find that $\mathsf q(\lambda)<0$ for $\lambda$ in the admissible range if and only if $\lambda>\lambda_\bullet$.

A minimization of $\mathsf Q(G,H)$ under the constraint
\[
\iz{\frac{G^2+H^2}{1-z^2}}=1
\]
reduces the problem to the identification of the ground state energy $\Lambda$ in the eigenvalue problem
\[
\left\{\begin{array}{l}
-\,\omega^2\!\(\(1-z^2\)G'\)'\!+\!\(1\!+\!\frac{4\omega^2}{(p-2)^2}\)\frac G{1-z^2}- \frac{2p\omega}{(p-2)^2} G-\frac{2a}{1-z^2}H=\frac\Lambda{1-z^2}G\,,\\[8pt]
-\,\omega^2\!\(\(1-z^2\)H'\)'\!+\!\(1\!+\!\frac{4\omega^2}{(p-2)^2}\)\frac H{1-z^2}-(p-1)\frac{2p\omega}{(p-2)^2}H-\frac{2a}{1-z^2}G=\frac\Lambda{1-z^2}H\,.
\end{array}\right.
\]
For given $(a,p)\in(0,1/2)\times(2,+\infty)$, the linear instability range $\mathcal I$ is the set of the parameters~$\lambda$ for which $\Lambda$ is negative. We know that
\[
(\lambda_\bullet,+\infty)\subset\mathcal I\subset(\lambda_\star,+\infty)
\]
but we do not even know whether $\mathcal I$ is an interval or not. Notice that for $a=1/2$, we find that $\zeta=1$ and $G=H$ is a good test function for any $\omega>0$: this means that there is symmetry breaking for any $\lambda>1/4$.

\subsection*{Acknowledgments}\begin{spacing}{0.9}{\small This research has been partially supported by the project \emph{EFI}, contract~ANR-17-CE40-0030 (D.B., J.D.) of the French National Research Agency (ANR), by the NSF grant DMS-1600560 (M.L.), and by the PDR (FNRS) grant T.1110.14F and the ERC AdG 2013 339958 ``Complex Patterns for Strongly Interacting Dynamical Systems - COMPAT'' grant (D.B.). The authors thank the referees for a careful reading which helped to remove some typos and improve the notations.
\\[2pt]\scriptsize
\copyright\,2019 by the authors. This paper may be reproduced, in its entirety, for non-commercial purposes.}\end{spacing}


\begin{thebibliography}{10}

\bibitem{Aubin-76}
{\sc T.~Aubin}, {\em Probl\`emes isop\'erim\'etriques et espaces de {S}obolev},
  J. Differential Geometry, 11 (1976), pp.~573--598.

\bibitem{BONHEURE2018}
{\sc D.~Bonheure, M.~Nys, and J.~Van~Schaftingen}, {\em Properties of ground
  states of nonlinear {S}chr\"{o}dinger equations under a weak constant
  magnetic field}, J. Math. Pures Appl. (9), 124 (2019), pp.~123--168.

\bibitem{MR3665549}
{\sc T.~Boulenger and E.~Lenzmann}, {\em Blowup for biharmonic {NLS}}, Ann.
  Sci. \'Ec. Norm. Sup{\'e}r. (4), 50 (2017), pp.~503--544.

\bibitem{Caffarelli-Kohn-Nirenberg-84}
{\sc L.~Caffarelli, R.~Kohn, and L.~Nirenberg}, {\em First order interpolation
  inequalities with weights}, Compositio Math., 53 (1984), pp.~259--275.

\bibitem{Catrina-Wang-01}
{\sc F.~Catrina and Z.-Q. Wang}, {\em On the {C}affarelli-{K}ohn-{N}irenberg
  inequalities: sharp constants, existence (and nonexistence), and symmetry of
  extremal functions}, Comm. Pure Appl. Math., 54 (2001), pp.~229--258.

\bibitem{MR1940370}
{\sc M.~Del~Pino and J.~Dolbeault}, {\em Best constants for
  {G}agliardo-{N}irenberg inequalities and applications to nonlinear
  diffusions}, J. Math. Pures Appl. (9), 81 (2002), pp.~847--875.

\bibitem{delPino20102045}
{\sc M.~del Pino, J.~Dolbeault, S.~Filippas, and A.~Tertikas}, {\em A
  logarithmic {H}ardy inequality}, Journal of Functional Analysis, 259 (2010),
  pp.~2045 -- 2072.

\bibitem{MR2966111}
{\sc J.~Dolbeault and M.~J. Esteban}, {\em Extremal functions for
  {C}affarelli-{K}ohn-{N}irenberg and logarithmic {H}ardy inequalities}, Proc.
  Roy. Soc. Edinburgh Sect. A, 142 (2012), pp.~745--767.

\bibitem{0951-7715-27-3-435}
\leavevmode\vrule height 2pt depth -1.6pt width 23pt, {\em Branches of
  non-symmetric critical points and symmetry breaking in nonlinear elliptic
  partial differential equations}, Nonlinearity, 27 (2014), p.~435.

\bibitem{Dolbeault06082014}
{\sc J.~Dolbeault, M.~J. Esteban, A.~Laptev, and M.~Loss}, {\em One-dimensional
  {G}agliardo--{N}irenberg--{S}obolev inequalities: remarks on duality and
  flows}, Journal of the London Mathematical Society, 90 (2014), pp.~525--550.

\bibitem{MR3784917}
\leavevmode\vrule height 2pt depth -1.6pt width 23pt, {\em Interpolation
  inequalities and spectral estimates for magnetic operators}, Ann. Henri
  Poincar\'{e}, 19 (2018), pp.~1439--1463.

\bibitem{doi:10.1063/1.5022121}
{\sc J.~Dolbeault, M.~J. Esteban, A.~Laptev, and M.~Loss}, {\em Magnetic
  rings}, Journal of Mathematical Physics, 59 (2018), p.~051504.

\bibitem{MR3570296}
{\sc J.~Dolbeault, M.~J. Esteban, and M.~Loss}, {\em Rigidity versus symmetry
  breaking via nonlinear flows on cylinders and {E}uclidean spaces}, Invent.
  Math., 206 (2016), pp.~397--440.

\bibitem{1703}
\leavevmode\vrule height 2pt depth -1.6pt width 23pt, {\em Symmetry and
  symmetry breaking: rigidity and flows in elliptic {PDE}s.}, Proc. Int. Cong.
  of Math. 2018, Rio de Janeiro, 3 (2018), pp.~2279--2304.

\bibitem{MR3612700}
{\sc J.~Dolbeault, M.~J. Esteban, M.~Loss, and M.~Muratori}, {\em Symmetry for
  extremal functions in subcritical {C}affarelli-{K}ohn-{N}irenberg
  inequalities}, C. R. Math. Acad. Sci. Paris, 355 (2017), pp.~133--154.

\bibitem{DELT09}
{\sc J.~Dolbeault, M.~J. Esteban, M.~Loss, and G.~Tarantello}, {\em On the
  symmetry of extremals for the {C}affarelli-{K}ohn-{N}irenberg inequalities},
  Adv. Nonlinear Stud., 9 (2009), pp.~713--726.

\bibitem{Erdos96}
{\sc L.~Erd{\H{o}}s}, {\em Rayleigh-type isoperimetric inequality with a
  homogeneous magnetic field}, Calculus of Variations and Partial Differential
  Equations, 4 (1996), pp.~283--292.

\bibitem{MR1708787}
{\sc P.~Exner, E.~M. Harrell, and M.~Loss}, {\em Optimal eigenvalues for some
  {L}aplacians and {S}chr{\"o}dinger operators depending on curvature}, in
  Mathematical results in quantum mechanics ({P}rague, 1998), vol.~108 of Oper.
  Theory Adv. Appl., Birkh{\"a}user, Basel, 1999, pp.~47--58.

\bibitem{Felli-Schneider-03}
{\sc V.~Felli and M.~Schneider}, {\em Perturbation results of critical elliptic
  equations of {C}affa\-relli-{K}ohn-{N}irenberg type}, J. Differential
  Equations, 191 (2003), pp.~121--142.

\bibitem{Ilyin}
{\sc V.~P. Il'in}, {\em Some integral inequalities and their applications in
  the theory of differentiable functions of several variables}, Mat. Sb.
  (N.S.), 54 (96) (1961), pp.~331--380.

\bibitem{MR1708811}
{\sc A.~Laptev and T.~Weidl}, {\em Hardy inequalities for magnetic {D}irichlet
  forms}, in Mathematical results in quantum mechanics ({P}rague, 1998),
  vol.~108 of Oper. Theory Adv. Appl., Birkh{\"a}user, Basel, 1999,
  pp.~299--305.

\bibitem{2018arXiv180506294L}
{\sc E.~{Lenzmann} and J.~{Sok}}, {\em {A sharp rearrangement principle in
  Fourier space and symmetry results for PDEs with arbitrary order}}, ArXiv
  e-prints,  (2018).

\bibitem{Lieb-83}
{\sc E.~H. Lieb}, {\em Sharp constants in the {H}ardy-{L}ittlewood-{S}obolev
  and related inequalities}, Ann. of Math. (2), 118 (1983), pp.~349--374.

\bibitem{MR2001882}
{\sc D.~Smets and M.~Willem}, {\em Partial symmetry and asymptotic behavior for
  some elliptic variational problems}, Calc. Var. Partial Differential
  Equations, 18 (2003), pp.~57--75.

\bibitem{zbMATH02502560}
{\sc B.~Sz{\"o}kefalvi-Nagy}, {\em {\"U}ber {I}ntegralungleichungen zwischen
  einer {F}unktion und ihrer {A}bleitung}, {Acta Sci. Math.}, 10 (1941),
  pp.~64--74.

\bibitem{Talenti-76}
{\sc G.~Talenti}, {\em Best constant in {S}obolev inequality}, Ann. Mat. Pura
  Appl. (4), 110 (1976), pp.~353--372.

\end{thebibliography}

\end{document}